\documentclass[10pt, reqno]{amsart}
\usepackage[english]{babel}
\usepackage[colorlinks,citecolor=green,linkcolor=red]{hyperref}
\usepackage{amsthm}
\usepackage{color}
\usepackage{mathrsfs}
\usepackage{amsmath}
\usepackage{mathtools}
\usepackage{amsfonts}
\usepackage{amssymb}
\usepackage{bm}
\usepackage{physics}
\usepackage{enumitem}
\usepackage{tikz-cd}
\usepackage{a4wide}
\usepackage{cleveref}
\usepackage{a4wide}
\usepackage{esint}
\usepackage{nicefrac}
\usepackage{yfonts}
\usepackage{dirtytalk}

\let\epsilon\varepsilon

\newtheorem{thm}{Theorem}
\newtheorem*{thm*}{Theorem}
\newtheorem{lem}[thm]{Lemma}
\newtheorem{prop}[thm]{Proposition}

\theoremstyle{definition}

\theoremstyle{remark}

\newcommand{\RCD}{{\mathrm {RCD}}}

\newcommand{\TestV}{{\mathrm {TestV}}}
\newcommand{\TestF}{{\mathrm {TestF}}}
\newcommand{\LIP}{{\mathrm {LIP}}}
\newcommand{\BV}{{\mathrm {BV}}}

\newcommand{\Gr}{{\mathrm {Gr}}}

\newcommand{\DIFF}{{\mathrm{D}}}
\newcommand{\lip}{{\mathrm {lip}}}

\DeclareMathOperator*{\dive}{div}
\DeclareMathOperator*{\essspan}{ess\,span}

\newcommand{\Id}{\mathrm{Id}}

\newcommand{\XX}{{\mathsf{X}}}
\newcommand{\dist}{{\mathsf{d}}}
\newcommand{\mass}{{\mathsf{m}}}

\newcommand{\capa}{{\mathrm {Cap}}}

\newcommand{\HH}{\mathcal{H}}
\newcommand{\RR}{\mathbb{R}}

\newcommand{\GG}{\mathcal{G}}

\newcommand{\defeq}{\mathrel{\mathop:}=}

\newcommand{\mres}{\mathbin{\vrule height 1.6ex depth 0pt width 0.13ex\vrule height 0.13ex depth 0pt width 1.3ex}}
\newcommand{\mressmall}{\mathbin{\vrule height 1.2ex depth 0pt width 0.11ex\vrule height 0.10ex depth 0pt width 1.0ex}}

\begin{document}
	
		\title[General chain rule for $\BV$ functions]{About the general chain rule for \\functions of bounded variation}
	
	\author[C. Brena]{Camillo Brena}
	\address{C.~Brena: Scuola Normale Superiore, Piazza dei Cavalieri 7, 56126 Pisa} 
	\email{\tt camillo.brena@sns.it}
	
	\author[N. Gigli]{Nicola Gigli}
	\address{N.~Gigli: SISSA, Via Bonomea 265, 34136 Trieste} 
	\email{\tt ngigli@sissa.it}

\begin{abstract}
We give an alternative proof of the general chain rule for functions of bounded variation (\cite{AmbDM}), which allows to compute the distributional differential of $\varphi\circ F$, where $\varphi\in \LIP(\RR^m)$ and $F\in\BV(\RR^n,\RR^m)$.
In our argument we build on top of recently established links between \say{closability of certain differentiation operators} and \say{differentiability of Lipschitz functions in related directions} (\cite{ABM23}): we couple this with the observation that \say{the map that takes $\varphi$ and returns the distributional differential of  $\varphi\circ F$ is closable} to conclude.

Unlike previous results in this direction, our proof can directly be adapted to the non-smooth setting of finite dimensional RCD spaces.
\end{abstract}
\maketitle
\section*{Introduction}
 A classical result about calculus for functions of bounded variation on the Euclidean space is the Vol'pert chain rule  (\cite{vol1985analysis,Vol_pert_1967}): if $F\in\BV_{loc}(\RR^n,\RR^m)$ and $\varphi\in C^1(\RR^m)\cap\LIP(\RR^m)$, it holds that $\varphi\circ F\in\BV_{loc}(\RR^n,\RR)$ and 
\begin{align}
	\DIFF(\varphi\circ F)\mres(\RR^n\setminus J_F)&=\nabla\varphi (F)\DIFF F\mres(\RR^n\setminus J_F),\label{introeq1} \\
	\DIFF(\varphi\circ F)\mres J_F&= \big({\varphi(F^+)-\varphi(F^-)}\big)\otimes \nu_{J_F} \HH^{n-1}\mres J_F\label{introeq2}.
\end{align}
Even though we refer to \cite{AFP00,GMSCartCurr} for the precise definition of the various objects appearing in \eqref{introeq1} and \eqref{introeq2} above, we explain here a bit of notation. Here $$\DIFF F=\frac{\DIFF F}{|\DIFF F|}{|\DIFF F|}$$ is the (polar decomposition of the) distributional of $F$, $J_F$ is the jump set of $F$ with normal $\nu_{J_F}$ and $F^\pm$ are the approximate jump values of $F$ (their order depends on the choice of the sign for $\nu_{J_F}$). Outside $J_F$, here and below, we implicitly take the precise representative of $F$.
\medskip

If we now take $\varphi\in \LIP(\RR^m)$ (not necessarily $C^1$) we can ask whether a similar formula holds for the composition $\varphi\circ F$ (which is still $\BV_{loc}$). Clearly, if this is the case, the formula has to be coherent with \eqref{introeq1} and \eqref{introeq2}. While \eqref{introeq2} causes no problem, we notice that, in general, \eqref{introeq1} makes no sense, as the image of $F$ may be essentially contained in the set of non-differentiability points of $\varphi$. This problem has been addressed and solved in \cite{AmbDM}, where the authors made the following crucial remark: we do not actually need the full differentiability of $\varphi$ at $F(x)$ for $|\DIFF F|$-a.e.\ $x\notin J_F$ to state \eqref{introeq1}. Indeed, the differential of $\varphi$ is only tested against the image of $\frac{\DIFF F}{|\DIFF F|}(x)$ at $F(x)$, so that a weaker form of differentiability is needed. More precisely, we only need the differentiability at $F(x)$ of the restriction of $\varphi$ to the affine subspace $$T_x^F\defeq F(x)+{\rm span}\bigg(\frac{\DIFF F}{|\DIFF F|}{}(x)\bigg)$$
for $|\DIFF F|$-a.e.\ $x\notin J_F$. We call $\nabla_{T_x^F}\varphi$, if exists, this differential. Notice that, by its very definition, the differentiability of the restriction of $\varphi$ to $T_x^F$ at $F(x)$ is equivalent to 
$$
\varphi(F(x)+w)-\varphi(F(x))=\nabla_{T_x^F}\varphi(F(x)) \,\cdot\, w+o(|w|)\qquad\text{for $w\in {\rm span}\bigg(\frac{\DIFF F}{|\DIFF F|}{}(x)\bigg) $}.
$$ 

In \cite{AmbDM} the authors proved that, for $\varphi\in\LIP(\RR^m)$ and $F\in\BV_{loc}(\RR^n,\RR^m)$, then the restriction of $\varphi$ to the affine subspace $T_x^F$ is differentiable at $F(x)$ for $|\DIFF F|$-a.e.\ $x\notin J_F$, that \eqref{introeq1} holds with $\nabla_{T_x^F}\varphi$ in place of $\nabla \varphi$ and that 
\eqref{introeq2} holds. As the proof of \eqref{introeq2} is folklore in what follows we address \eqref{introeq1} and, in particular, the issue of differentiability of Lipschitz functions.

The starting point of the proof of \cite{AmbDM} is to notice that it is enough to show that for $v\in\RR^n\setminus\{0\}$ fixed arbitrarily, for $|\DIFF F|$-a.e.\ $x\notin J_F$ then the restriction of $\varphi$ to the subspace 
$$
T^{F,v}_x\defeq F(x)+{\rm span}\bigg(\frac{\DIFF F}{|\DIFF F|}(x)\,\cdot\, v\bigg)
$$
is differentiable at $F(x)$ and that, calling $\nabla_{T^{F,v}_x}\varphi$ this differential, it holds that

\begin{equation}\label{fwdnj}
	\DIFF(\varphi\circ F)\,\cdot v\mres (\RR^n\setminus J_F)=\nabla_{T^{F,v}_x}\varphi(F(x))\DIFF F \,\cdot\,v\mres (\RR^n\setminus J_F),
\end{equation}
which is precisely \eqref{introeq1} tested against $v$. Then the claim follows by classical analysis of Lipschitz functions. We are going to take advantage of this last remark and in what follows we discuss about the proof of this fact.

The proof of \cite{AmbDM} worked by slicing on an hyperplane $H^v$ orthogonal to $v$: they first proved the result for functions $\BV_{loc}(\RR,\RR^m)$ (given by $F^v_x( t)\defeq F(x+t v)$ for $x\in H^v$) and then concluded using the well known superposition formula $$\DIFF F\,\cdot\, v=\int_{H^v}\DIFF F_x^v\dd{\HH^{n-1}(x)}.$$
Motivated by studies in the non-smooth setting - see below - where slicing procedures seems harder to get, we propose here a different strategy.

First, we can assume that $F$ has a Borel inverse. We can indeed replace $F$ by $$\RR^n\ni x\mapsto (F(x),x)\in\RR^{m+n}$$ and $\varphi$ by $$\RR^{m+n}\ni(x,y)\mapsto \varphi(x)\in\RR$$
and it turns out that the claim can be proved for this new couple of maps, that we still call $F$ and $\varphi$.
We consider now 
$$
w\defeq \bigg(\frac{\DIFF F}{|\DIFF F|}\,\cdot\, v\bigg)\circ F^{-1}\qquad\text{and}\qquad\mu\defeq F_*(|\DIFF F|\mres K),
$$
where $K\subseteq \RR^n\setminus J_F$ is any compact set. 
Then we define the differentiation operator $D$
 $$
C^1(\RR^m)\cap\LIP(\RR^m)\ni \psi\mapsto D(\psi)\defeq \nabla \psi\,\cdot\,w\in L^\infty(\mu).
$$
The next step is to verify that this operator enjoys the following closability property: for every $\{\psi_k\}_k\subseteq C^1(\RR^m)\cap\LIP(\RR^m)$ equi-bounded and equi-Lipschitz functions converging pointwise to $\psi\in\LIP_b(\RR^m)$, then $D(\psi_k)\rightarrow \ell$ in the weak* topology of $L^\infty(\mu)$, for some $\ell\in L^\infty(\mu)$. This follows from an approximation result, the Vol'pert chain rule formula and the closability of the map $\varphi\mapsto\DIFF(\varphi\circ F)$. The key computation is as follows: if $h\in L^1(\mu)$, then, neglecting evident regularity issues (here the approximation procedure), 
\begin{align*}
	\int_{\RR^m}h\nabla\psi_k\,\cdot\,w\dd{\mu}&=\int_{\RR^n}\chi_K h (F)\nabla\psi_k(F)\dd{{\DIFF F}\,\cdot\,v}=\int_{\RR^n}\big(\chi_K h(F) v\big)\dd{\DIFF(\psi_k\circ F)}\\
	&=-\int_{\RR^n} \psi_k\circ F \dive \big(\chi_K h (F) v\big)\dd{\mathcal L^n}\rightarrow-\int_{\RR^n} \psi\circ F \dive \big(\chi_K h (F) v\big)\dd{\mathcal L^n}\\
	&=\int_{\RR^n}\big(\chi_K h (F) v\big)\dd{\DIFF(\psi\circ F)}=\int_{\RR^n}\chi_K h (F)\frac{\DIFF(\psi\circ F)}{|\DIFF(\psi\circ F)|}\,\cdot\,v\dv{|\DIFF(\psi\circ F)|}{|\DIFF  F|}\dd{|\DIFF F|}\\
	&=\int_{\RR^m}h\, \bigg(\frac{\DIFF(\psi\circ F)}{|\DIFF(\psi\circ F)|}\,\cdot\,v\dv{|\DIFF(\psi\circ F)|}{|\DIFF  F|}\bigg)\circ F^{-1}\dd{\mu}.
\end{align*}

This closability property allows us to apply Theorem \ref{AMthm} (which is a restatement of results contained in \cite{AM16,ABM23}) which shows that Lipschitz functions are differentiable in direction $w(x)$ for $\mu$-a.e.\ $x\in\RR^m$ and, reading back this information in $\RR^n$ we conclude the proof of differentiability and then \eqref{fwdnj} follows. 
\medskip

The purpose of this note is to prove the general chain rule in the realm of $\RCD(K,N)$ spaces (\cite{AmbrosioMondinoSavare13-2}, \cite{Gigli12} after  \cite{Sturm06I,Sturm06II}, \cite{Lott-Villani09} -  see also the surveys \cite{AmbICM,Villani2017,gigli2023giorgi}), which are metric measure spaces satisfying, in a synthetic sense, a bound from below for the Ricci curvature and a bound from above for the dimension. Indeed, starting from the seminal papers \cite{MIRANDA2003,amb00,amb01} (see e.g.\ \cite{DiM14a} and references therein) for what concerns general metric measure spaces and following \cite{ambrosio2018rigidity,bru2019rectifiability,bru2021constancy, BGBV} for what concern $\RCD(K,N)$ spaces, it is clear that a reasonable theory can be developed for functions of bounded variation on non-smooth spaces.
Here we state the general chain rule for $\RCD(K,N)$ spaces (Theorem \ref{gcthm}). For the interpretation of the various objects appearing, we refer to the rest of the paper, but the meaning such objects can also be  deduced by comparison with \eqref{introeq1} and \eqref{introeq2}.   
\begin{thm*}
	Let $(\XX,\dist,\mass)$ be an $\RCD(K,N)$ space, let $F\in\BV_{loc}(\XX)^m$ and let $\varphi\in \LIP(\RR^m)$.
	Then  
	\begin{align*}\notag
		\nu_{\varphi\circ F}|\DIFF(\varphi\circ F)|\mres (\XX\setminus J_F)&=\nabla_V\varphi(F) \nu_F {|{\DIFF F}|}\mres (\XX\setminus J_F),\\
		\nu_{\varphi\circ F}|\DIFF(\varphi\circ F)|\mres J_F&=\frac{\varphi(F^r)-\varphi(F^l)}{|F^r-F^l|} \nu_F^J {|{\DIFF F}|}\mres  J_F,
	\end{align*}
	where we implicitly state that for $|\DIFF F|$-a.e.\ $x\notin J_F$, $\varphi$ is differentiable at $F(x)$ with respect to $V$, which is the image, in a suitable sense, of $\nu_F$.
\end{thm*}

\medskip

We decided to give directly the proof of the result in the more general framework of $\RCD(K,N)$ spaces (as $(\RR^N,\dist_e,\mathcal L^N)$ is $\RCD(K,N)$). The reader only interested in the case of $\RR^n$ can just read Section \ref{difflip} and the proof of Lemma \ref{mainlemma}. In the latter, we fill the gaps of the proof sketched above. Then only thing to take into into account is a change of notation for what concerns the differential of the restriction of $\varphi$ (see Section \ref{difflip}) and the polar decomposition of the distributional differential, as $\frac{\DIFF F}{|\DIFF F|}|\DIFF F|$ becomes $\nu_F|\DIFF F|$.

\section{Preliminaries}\label{sectprel}
For the sake of brevity, we directly assume the reader to be familiar with the setting of $\RCD$ spaces, (\cite{AmbICM,Villani2017,gigli2023giorgi} and references therein) and with the theory of functions of bounded variation on metric measure spaces (\cite{MIRANDA2003, amb00,amb01,ambmirpal04,ambrosio2018rigidity, bru2019rectifiability,bru2021constancy,BGBV}).
A summary of the prerequisites can be found in the preliminaries section of \cite{BGBV}.
\medskip

We point out that we are going to use the following definition for test vector fields,
\begin{equation}\label{defntest}
	\TestV(\XX)\defeq\left\{ \sum_{i=1}^n f_i\nabla g_i : f_i\in S^2(\XX)\cap L^\infty(\mass),g_i\in\TestF(\XX)\right\},
\end{equation}
where, 
\begin{equation}\notag
	\TestF(\XX)\defeq\left\{f\in\LIP(\RR)\cap L^\infty(\mass)\cap D(\Delta): 	\Delta f\in H^{1,2}(\XX)\cap L^\infty(\mass)\right\}.
\end{equation}
This is a slight difference with the literature that we just recalled, and we adopted this new definition for simplicity of notation in Theorem \ref{captangent} and related material, but we point out that this difference is harmless.

We are going to consider only \emph{finite dimensional} $\RCD$ spaces, so that when we write $\RCD(K,N)$ we always assume $N<\infty$.

\medskip

We recall below the main notions, in a simplified version, as will be the starting point for the theorem proved in this note. Again, for further details, we refer for example to \cite{BGBV}.
To state our results, it is required the notion of (capacitary) normed module,  \cite{Gigli14,debin2019quasicontinuous}; what is needed here is also recalled in \cite[Section 1.3]{bru2019rectifiability}. 

First, we recall the definition of total variation of a (vector valued) locally integrable function: if $(\XX,\dist,\mass)$ is an $\RCD(K,N)$ space and $F\in L^1(\XX,\RR^n)$, we define, for $A\subseteq\XX$ open
	\begin{equation}\notag
	|{\DIFF F}|(A)\defeq \inf \left\{\liminf_k \int_A \Vert (\lip(F_{i,k}))_{i=1,\dots,n}\Vert_e\dd{\mass}\right\}
\end{equation}
where the infimum is taken among all sequences $\{F_{i,k}\}_k\subseteq\LIP_{loc}(A)$ such that $F_{i,k}\rightarrow F_i$ in $L_{loc}^1(A,\mass)$ for every $i=1,\dots,n$.  If $\abs{\DIFF F}(A')<\infty$ for every relatively compact $A'\subseteq\XX$, we say that $F\in\BV_{loc}(\XX,\RR^n)=\BV_{loc}(\XX)^n$ and it turns out that $|\DIFF F|$ is the restriction to open sets of a  Radon measure that we still denote by $|\DIFF F|$ and which satisfies $|\DIFF F|\ll\capa$.  If moreover $|\DIFF F(\XX)|<\infty$, then we say that $F\in\BV(\XX,\RR^n)=\BV(\XX)^n$. 

If $F=(F_1,\dots,F_n)\in\BV(\XX)^n$, there exists a unique, up to $|\DIFF F|$-a.e.\ equality, $\nu_F\in L^0_\capa(T\XX)^n$ such that $|\nu_F|=1\ |\DIFF F|$-a.e.\ and 
\begin{equation}\label{intbyparts}
	\sum_{j=1}^n\int_\XX F_j{\rm div}(v_j)\dd{\mass}=-\int_\XX v\,\cdot\,\nu_F\dd{|\DIFF F|}\quad\text{ for every $v=(v_1,\dots,v_n)\in\TestV(\XX)^n$.}
\end{equation}
Here and in the sequel, we identify quasi-continuous vector fields with the the trace they leave on the $\capa$ or $|\DIFF F|$-tangent module, and we do the same for functions. Also, for $v\in\TestV(\XX)$, it is possible that $\dive v\notin L^\infty(\mass)$. However, the formula above is well posed, interpreting, for $f\in\BV(\XX)$,
$$
\int_\XX f\dive v\dd{\mass}=\lim_k \int_\XX( f\vee k)\wedge k \dive v\dd{\mass}.
$$ 
By an immediate locality procedure, if $F=(F_1,\dots,F_n)\in\BV_{loc}(\XX)^n$, there exists a a unique, up to $|\DIFF F|$-a.e.\ equality, $\nu_F\in L^0_\capa(T\XX)^n$ such that $|\nu_F|=1\ |\DIFF F|$-a.e.\ and \eqref{intbyparts} holds restricted to $n$-tuples of compactly supported test vector fields. Notice that the existence of good cut-off functions (\cite{Mondino-Naber14}) gives the existence of sufficiently many compactly supported test vector field to have uniqueness of $\nu_F$ also in this case. 

It is customary to define, for $f\in \BV_{loc}(\XX)$, the  lower and upper representatives $f^\wedge\le f^\vee$, as well as a precise representative $\bar{f}=\frac{f^\wedge+f^\vee}{2}$. We define the jump set $J_f\defeq\left\{x:f^\wedge< f^\vee\right\}$ (pay attention to the change of notation, what we call here $J_f$ was called $S_f$ in \cite{BGBV}). It is clear that if $\varphi$ is Lipschitz, then $J_{\varphi\circ f}\subseteq J_f$. It is well known that $|\DIFF f|\mres (\XX\setminus J_f)$ does not charge jump sets of functions of bounded variation.
 If $F\in\BV_{loc}(\XX)^n$, then we set $J_F\defeq\bigcup_{i=1}^n J_{F_i}$. We also proved that for $|\DIFF F|$-a.e.\ $x\in J_F$, there exists a unique couple (up to their order) $F^r(x),F^l(x)$ such that, for a suitable set of finite perimeter $E_x$, it holds that $x\in J_{E_x}$ and 
 $$
 \lim_{r\searrow 0}\frac{1}{\mass(B_r(x)\cap E_x)}\int_{B_r(x)\cap E_x}| F-F^l(x)|\dd{\mass}= \lim_{r\searrow 0}\frac{1}{\mass(B_r(x)\setminus E_x)}\int_{B_r(x)\setminus E_x}| F-F^r(x)|\dd{\mass}=0.
 $$
 Moreover, the map $$J_F\rightarrow\RR^n\qquad x\mapsto (F^l(x),F^r(x))$$
 is $|\DIFF F|$-measurable. Also, outside $J_F$, we set the precise representative as the vector whose components are the precise representatives of the components of $F$ and \emph{we implicitly take this representative} for $F$, always. We also proved that there exists a unique, up to $|\DIFF F|\mres J_F$ equality $\capa$-vector field $\nu_F^J$ such that 
$$(F_i^r-F_i^l)\nu_F^J=(F_i^\vee-F_i^\wedge)\nu_{F_i}\qquad|\DIFF F|\mres J_F\text{-a.e.\ for every $i=1,\dots,n$}.$$

As a notation, if $\nu,\nu'\in L^\infty_\capa(T\XX)^n$ and $\mu,\mu'$ are two Radon measure absolutely continuous with respect to $\capa$, we write $\nu\mu=\nu'\mu'$ if, for every $v\in L^\infty_\capa(T\XX)^n$ with compact support it holds
$$
\int_\XX v\,\cdot\,\nu\dd{\mu}= \int_\XX v\,\cdot\,\nu'\dd{\mu'}. 
$$
By density, the equality above is satisfied for every $v\in L^\infty_\capa(T\XX)^n$ with compact support if and only if it is satisfied for every $v\in \TestV(\XX)^n$ with compact support. 
We use this notation to state our calculus rules. For example, if $F\in\BV_{loc}(\XX)^n$ and $\varphi\in C^1(\RR^n)\cap\LIP(\RR^n)$, we have proved 
\begin{align*}
	\nu_{\varphi\circ F}|\DIFF(\varphi\circ F)|\mres (\XX\setminus J_F)&=\nabla\varphi(F)\nu_F|\DIFF F|\mres (\XX\setminus J_F)\\
	\nu_{\varphi\circ F}|\DIFF(\varphi\circ F)|\mres  J_F&=\frac{\varphi(F^r)-\varphi(F^l)}{|F^r-F^l|}\nu_F^J|\DIFF F|\mres J_F.
\end{align*}
\section{Differentiability of Lipschitz functions}\label{difflip}
Given $\varphi\in\LIP(\RR^m,\RR^l)$, we say that $\varphi$ is differentiable at $x$ with respect to $V\in\Gr(\RR^m)$ if there exists a linear map $\nabla_V\varphi(x):V\rightarrow\RR^l$ such that 
$$ \varphi(x+v)=\varphi(x)+\nabla_V\varphi(x)\,\cdot\, v+ o(|v|)\qquad\text{for $v\in V$}.$$
If $v\in\RR^m$, we say that $\varphi$ is differentiable at $x$ in direction $v$ if $\varphi$ is differentiable at $x$ with respect $\langle v\rangle$.
Notice that every $\varphi$ is differentiable with respect to $\{0\}$ at any point of $\RR^m$.

\medskip

We are going to exploit in a crucial way the following result, which is a restatement of results contained in \cite{AM16,ABM23} (see in particular, \cite[Theorem 1.1]{AM16} and \cite[Theorem 1.1]{ABM23}). We refer the reader to these references for the definition of $V(\mu,\,\cdot\,)$, the decomposability bundle associated to the Radon measure $\mu$, as we are not going to use this notion elsewhere.
\begin{thm}\label{AMthm}
	Let $ v\mu$ be a $m$-vector valued measure on $\RR^m$, where $v\in L^\infty( \mu)$ and $\mu$ is finite. Then the following assertions are equivalent.
	\begin{enumerate}[label=\roman*)]
		\item $v(x)\in V(\mu,x)$ for $\mu$-a.e.\ $x$.
	\item Every Lipschitz function is differentiable in direction $v(x)$ for $\mu$-a.e.\ $x$. 
	\item The operator $$D: C^1(\RR^m)\cap\LIP_b(\RR^m)\rightarrow L^\infty(\mu)\qquad \varphi\mapsto \nabla\varphi\,\cdot\, v$$ is closable, in the sense that if $\{\varphi_k\}_k\subseteq C^1(\RR^m)\cap\LIP_b(\RR^m)$ is a sequence of equi-bounded and equi-Lipschitz  functions converging pointwise to $\varphi\in\LIP_b(\RR^m)$, then $D(\varphi_k)\rightarrow \ell$ in the weak$^*$ topology of $L^\infty(\mu)$, for some $\ell \in L^\infty(\mu)$.
	\end{enumerate}
	If any (hence all) of the items above holds, if $\ell$ is as in item $iii)$ for $\varphi$, then 
\begin{equation}\label{whoell}
	\ell(x)=\nabla_{v(x)}\varphi(x)\,\cdot\,v(x)\qquad\text{for }{\mu}\text{-a.e.\ }x.
\end{equation}
\end{thm}

In our approximation arguments, we are going to need the following result, that is extracted from \cite[Corollary 8.3]{AM16}.
\begin{lem}\label{ottopuntotre}
Let $\varphi\in \LIP(\RR^m)$ and let $\mu$ be a finite measure on $\RR^m$. Assume also that $x\mapsto v(x)$ is a bounded Borel map such that for $\mu$-a.e.\ $x$, $\varphi$ is differentiable in direction $v(x)$ at $x$. Then there exists a sequence $\{\varphi_k\}_k\subseteq C^1(\RR^m)\cap\LIP(\RR^m)$ such that $\varphi_k\rightarrow \varphi$ uniformly, the global Lipschitz constant of  $\varphi_k$ converges to the global Lipschitz constant of  $\varphi$ and finally $$\nabla_{v(x)}\varphi_k(x)\,\cdot v(x)\rightarrow \nabla_{v(x)}\varphi(x)\,\cdot v(x)\qquad\mu\text{-a.e.\ $x$}.$$
\end{lem}

\section{Technical results}\label{secttech}
Recall first our working definition of test vector fields in \eqref{defntest}.
\begin{thm}\label{captangent}
	Let $(\XX,\dist,\mass)$ be an $\RCD(K,N)$ space of essential dimension $n$. Then there exists a partition of $\XX$ made of countably many bounded Borel sets $\{A_k\}_{k}$ such that for every $k$ there exist $n(k)$ with $0\le n(k)\le n$ and $\big\{v_1^k,\dots,v^k_{n(k)}\big\}\subseteq\TestV(\XX)$ with bounded support which is an orthonormal basis of $L^0_\capa(T\XX)$ on $A_k$, in the sense that $$ v_i\,\cdot\,v_j=\delta_i^j\qquad\capa\text{-a.e.\ on }A_k$$ and for every $v\in L^0_\capa(T\XX)$ there exist $g_1,\dots,g_{n(k)}\in L^0(\capa)$ such that 
	$$v=\sum_{i=1}^{n(k)}g_i v_i^k\qquad\capa\text{-a.e.\ on }A_k,$$
	where, in particular,  $$ g_i=v\,\cdot\,v_i^k\qquad\capa\text{-a.e.\ on }A_k.$$
	
	Here we implicitly state that if $n(k)=0$ then for every $v\in L^0_\capa(T\XX)$ we have $v=0\ \capa$-a.e.\ on $A_k$.
\end{thm}
\begin{proof}
	First, we recall that in \cite{debin2019quasicontinuous} the capacitary tangent module was defined as follows.
	There exists a unique (in a suitable sense) couple $(L^0_\capa(T\XX),\tilde\nabla)$, where $L^0_\capa(T\XX)$ is a $L^0(\capa)$-normed $L^0(\capa)$-module and $\tilde\nabla:	\TestF(\XX) \rightarrow L^0_\capa(T\XX)$ is a linear operator such that:
	\begin{enumerate}[label=\roman*)]
		\item
		$|{\tilde\nabla f}|=(\abs{\nabla f}) \ \capa$-a.e.\ for every $f\in\TestF(\XX)$, where at the right hand side we take the quasi-continuous representative (\cite{debin2019quasicontinuous}),
		\item
		the set $\left\{\sum_{n} \chi_{E_n}\tilde \nabla f_n\right\}$, where $\{f_n\}_n\subseteq\TestF(\XX)$ and $\{E_n\}_n$ is a Borel partition of $\XX$ is dense in $L^0_\capa(T\XX)$.
	\end{enumerate}
	It is easy to show that in item $ii)$ above we can replace $\TestF(\XX)$ with a countable subset, say $\{f_k\}_{k\in\mathbb N}$. 
This is due to the fact that $H^{1,2}_C(T\XX)$ is separable, so that we can take a countable dense subset of $\TestV(\XX)$ in the $W^{1,2}_C(T\XX)$ topology. Set now $v_k\defeq\tilde\nabla{f_k}$ and $D\defeq\{v_k\}_k$.

Consider now the (countable) sequence of $\capa$-a.e.\ defined functions $$F_I\defeq\det\left(v_i\,\cdot\,v_j\right)_{i,j\in I}$$
where $I$ ranges over all finite subsets of $\mathbb N$.
Notice that if $\abs{I}> n$ then the fact that $L^2(T\XX)$ has dimension $n$ and basic linear algebra yield that $F_I=0\ \mass$-a.e.\ hence $F_I=0\ \capa$-a.e.\ because $F_I$ is quasi-continuous.
We set then for $i\in\mathbb N$, $i\ge 1$ $$A_i\defeq\bigcup_{\abs{I}=i}\left\{F_I\ne 0\right\}\bigcap_{\abs{J}\ge i+1}\left\{ F_J=0\right\} $$
and $A_0\defeq\XX\setminus\cup_{i\ge 1}A_i$.
Notice $\XX=A_0\cup\cdots\cup A_n$ as a disjoint union. 

Notice now that we can show, by density, that for every $v\in L^0_\capa(T\XX)$ we have $v=0\ \capa$-a.e.\ on $A_0$.
Then, by definition, we can write $A_i=\bigcup_I A_i^I$ countable (disjoint) union, where on $A_i^I$ we have that $F_I\ne 0\ \capa$-a.e.\ and $F_J=0\ \capa$-a.e.\ if $\abs{J}>\abs{I}$.
We claim now that $\{v_k\}_{k\in I}$ is a basis of $ L^0_\capa(T\XX)$ on $A_i^I$, in the sense that for every $v\in L^0_\capa(T\XX)$ there exists $\{g_k\}_{k\in I}$ such that $$v=\sum_{k\in I} g_k v_k\quad\capa\text{-a.e.\ on } A_i^I$$ and that if for some $\{g_k\}_{k\in I}\subseteq L^0(\capa)$ we have that $\sum_{k\in I}g_k v_k=0\ \capa$-a.e.\ then $g_k=0\ \capa$-a.e.\ on $A_i^I$ for every $k\in I$. 

The second claim follows by basic linear algebra: indeed if $\sum_{k\in I}g_k v_k=0\ \capa$-a.e.\ then we have in particular $$\sum_{k,h\in I} g_h g_k v_h\,\cdot\, v_k=0\quad \capa\text{-a.e.}$$ and this implies $ g_k=0\ \capa$-a.e.\ on $A^I_i$ for every $k\in I$ as $F_I\ne 0\ \capa$-a.e.\ on $A^I_i$.

We show now the first claim. This is again basic linear algebra together with a simple density argument. Take any $w$ in $L^0_\capa(T\XX)$, then a  density-continuity argument and the fact that $F_{I\cup \{{k}\}}=0\ \capa$-a.e.\ on $A_i^I$ for every $k$ show that if we set (with an abuse) $v_{\bar{k}}\defeq w$ we have that $$F_{I\cup\{\bar{k}\}}=0\quad\capa\text{-a.e.\ on }A_i^I.$$ In particular, as $F_I\ne 0\ \capa$-a.e.\ on $A_i^I$, we have that $v_{\bar{k}}\,\cdot v_l=\sum_{j\in I} g_j v_j \,\cdot\, v_l$ $\capa$-a.e.\ on $A_i^I$ for $l\in\ I\cup\{\bar{k}\}$ where $\{g_j\}_j\subseteq L^0(\capa)$ (as they are the unique solution of a linear system with coefficients in $ L^0(\capa)$). This immediately implies $$\bigg|v_{\bar{k}}-\sum_{j\in I} g_j v_j \bigg| ^2=0\quad\capa\text{-a.e.\ on } A_i^I.$$

We do now a further decomposition of the sets $A^I_i$.
First, we orthogonalize the basis $\{v_k\}_{k\in I}$ by means of a Gram-Schmidt procedure as follows. Assume for simplicity $I=\{1,2,\dots,m\}$. We set recursively
$$v_k'\defeq c_k^k v_k+ \sum_{l=1}^{k-1} c_l^k v'_l\quad\text{for }k=1,\dots,m,$$
where $\{c_l^k\}_{1\le l\le k\le m}$ are defined as 
\begin{equation}\notag
	c_l^k\defeq
	\begin{dcases}
		\prod_{j=1}^{k-1} v'_j\,\cdot\,v'_j\quad&\text{if $l=k$},\\
		-\frac{v_k\,\cdot\,v'_l}{v'_l\,\cdot\,v'_l} \prod_{j=1}^{k-1} v'_j\,\cdot\,v'_j\quad&\text{if $l<k$.}
	\end{dcases}
\end{equation}
Notice that $\{c_l^k\}\subseteq S^2(\XX)\cap L^\infty(\mass)$ and that $\{v'_k\}_{k\in I}$ is still a basis on $A_i^I$ in the sense described above. Also, $v_h'\,\cdot\,v_k'=0\ \capa$-a.e.\ on $A_i^I$ if $h\ne k, h,k\in I$.

If $\epsilon>0$, we set then $(A^I_i)_\epsilon\defeq A^I_i\cap \left\{\abs{v'_k}>\epsilon\ \text{for every }k\in I\right\}$, notice that $A_i^I=\bigcup_{\epsilon>0}(A^I_i)_\epsilon$ and we can write such union as a countable union. We rescale now the basis writing 
$$v''_k\defeq\frac{1}{\epsilon\vee \abs{v'_k}}v'_k$$
and this allows us to conclude the proof.
\end{proof}

In this paper if $\nu=(\nu_1,\dots,\nu_m)\in L^0_\capa(T\XX)^m$ and $v\in L^0_\capa(T\XX)$, we write 
$$
\nu\,\cdot\,v\defeq(\nu_i\,\cdot\, v)_i\in\RR^m.
$$

\begin{lem}\label{essspan}
	Let $(\XX,\dist,\mass)$ be an $\RCD(K,N)$ space, $\mu\ll\capa$ a finite Borel measure and $\nu\in L^0_\capa(T\XX)^m$. Then there exists unique (up to $\mu$-a.e.\ equality)  $\mu$-measurable map
	$$G:\XX\rightarrow\Gr(\RR^m)$$
	satisfying
	\begin{enumerate}[label=\roman*)]
		\item for every $v\in L^0_\capa(T\XX)$, $$ \nu \,\cdot\,v\in G\qquad\text{$\mu$-a.e.}$$
		\item if $G':\XX\rightarrow\Gr(\RR^m)$ is another map satisfying the requirement $i)$, then
		$$ G\subseteq G'\qquad\text{$\mu$-a.e.}$$
	\end{enumerate}
\end{lem}
We call the map $G$ given by the lemma above $\mu-\essspan\nu$.
\begin{proof}
	First notice that uniqueness of $G$ trivially follows from item $ii)$.
	
	Fix for the moment a set $A$ as in the decomposition given by Theorem \ref{captangent} and an orthonormal basis of $L^0_\capa(T\XX)$ on $A$, say $\{v_1,\dots,v_k\}$. The map $$A\ni x\mapsto \langle \left\{\nu  \,\cdot \,v_i(x)\right\}_{i=1,\dots,k}\rangle\in\Gr(\RR^m)$$
	is $\mu$-measurable. We then define $G$ equals this map on $A$ and then define $G$ $\mu$-a.e.\ on $\XX$ with a gluing argument. 
	
	We show now that $G$ satisfies the desired properties. It is sufficient to fix a set $A$ and vector fields $\{v_1,\dots,v_k\}$ as above and prove the claims on $A$. Item $i)$ follows from the fact that $\{v_1,\dots,v_k\}$ is a basis of $L^0_\capa(T\XX)$ on $A$. For what concerns item $ii)$, take $G'$ satisfying item $i)$. In particular, $\mu$-a.e.\ $\nu  \,\cdot \,v_i\in G$ for every $i=1,\dots,k$ so that $\mu$-a.e.\ $G\subseteq G'$.
\end{proof}

In the introduction we have seen that we need to work with functions of bounded variation that have a Borel inverse. On the Euclidean space it was easy to reduce ourselves to this situation, by adding the identity map after the function. On $\RCD(K,N)$ spaces the situation is slightly more complicated as we do not have a global, smooth chart. We solve this issue partitioning $\XX$, up to a negligible set, in subsets on which we have a reasonable notion of \say{chart} and this is the content of the following lemma.
\begin{prop}\label{charts}
	Let $(\XX,\dist,\mass)$ be an $\RCD(K,N)$ space of essential dimension $n$ and let $F\in\BV(\XX)^m$. Then there exists a countable collection $\{(G_i,\Psi_i,B_i)\}_i$  such that 
	\begin{enumerate}[label=\roman*)]
		\item $\{G_i\}_i$ is a collection of pairwise disjoint Borel subset of $\XX$ satisfying $$|\DIFF F|\bigg(\XX\setminus\bigcup_i G_i\bigg)=0,$$
		\item $\{B_i\}_i$ is a collection of Borel subsets of $\RR^{n}$,
	\item for every $i$, $$\Psi_i:G_i\rightarrow B_i$$
	is an invertible Borel map with Borel inverse
	$$\Psi_i^{-1}:B_i\rightarrow G_i,$$ 
	\item for every $i$, $\Psi_i$ is the restriction of some $\tilde{\Psi}_i\in\BV(\XX)^n$ to $G_i$ with $J_{\tilde{\Psi}_i}\cap G_i=\emptyset$.
	\end{enumerate}
\end{prop}
\begin{proof}
As $|\DIFF F|\le |\DIFF F_1|+\cdots+|\DIFF F_m|$, there is no loss of generality in assuming $m=1$, so write $f=F_1$.  

Split $\XX$ in the disjoint union $A_f\cup C_f\cup J_f$, according to the decomposition $$ |\DIFF f|=|\DIFF f|^{ac}+|\DIFF f|^{C}+|\DIFF f|^j$$
in absolutely continuous, Cantor and jump part. Then the claim on the portion $J_f$ follows from the rectifiability result of \cite{bru2019rectifiability}. The claim on the portion $A_f$ follows from the rectifiability of $(\XX,\dist,\mass)$, e.g.\ \cite{Mondino-Naber14} or \cite{BruPasSem20}. As the maps providing rectifiability are bi-Lipschitz, an application of McShane extension Theorem shows that item $iv)$ can be satisfied.

We treat now the part $C_f$. First, we define the subgraph of $f$ as $$\GG_f\defeq\{(x,t)\in\XX\times\RR:t<f(x)\}.$$
By \cite[Theorem 5.1]{Ambrosio-Pinamonti-Speight15} we obtain that $\GG_f$ is a set of locally finite perimeter and that, if $\pi:\XX\times\RR\rightarrow\XX$ denotes the projection onto the first factor, it holds 
$$|\DIFF f|\le \pi_*|\DIFF\chi_{\GG_f}|.$$
By the rectifiability result of \cite{bru2019rectifiability} again, there exists a countable collection $\{\hat{C_i}\}_i$ of pairwise disjoint Borel subsets of $\partial^*\GG_f\subseteq\XX\times\RR$ such that $$|\DIFF\chi_{\GG_f}|\bigg((\XX\times\RR)\setminus \bigcup_i\hat{C}_i\bigg)=0$$
and for every $i$ there exists a map 
$$
\Phi_i: \hat{C_i}\rightarrow \hat{B}_i\subseteq\RR^n
$$
which is bi-Lipschitz onto its image. To our aim, there is no loss of generality in assuming that for every $i$, $\hat{C_i}\subseteq (C_f\cap\left\{f\in\RR\right\})\times\RR$ (we use also \cite[Lemma 3.2]{kinkorshatuo}). 
We recall that (see e.g.\ \cite[Lemma 2.11]{ABPrank}) $$(x,t)\in\partial^*\GG_f\quad\Rightarrow\quad t\in[f^{\wedge}(x,),f^{\vee}(x)].$$
We set now for every $i$, $C_i\defeq\pi(\hat{C}_i)$ and $\Psi_i\defeq \Phi_i\circ (\Id, f)_{|C_i}$, and it is easy to show that this assignment satisfies the request in item $iii)$.  

We now show item $iv)$, up to removing from $C_i$ the $|\DIFF f|$-negligible subset $J_{\tilde\Psi_i}\cap C_i$ (the fact that $J_{\tilde\Psi_i}\cap C_i$ is $|\DIFF f|$-negligible follows from \eqref{123r31} below and the fact that $|\DIFF f|\mres(\XX\setminus J_f)$ does not charge jump sets of functions of bounded variation). We first use McShane extension Theorem for $\Phi_i$ to obtain an $L$-Lipschitz function $\tilde{\Phi}_i$ and set $\tilde{\Psi}_i\defeq\tilde{\Phi}_i\circ(\Id,f)$. Notice that if $g\in\LIP_{loc}(\XX)$ it holds that 
$\lip(\tilde{\Phi}_i\circ(\Id,g))\le L(\lip(g)+1)$, therefore an approximation argument yields that $\DIFF \tilde{\Psi}_i\in\BV_{loc}(\XX)^n$ with \begin{equation}\label{123r31}
	|\DIFF \tilde{\Psi}_i|\le L(|\DIFF f|+\mass)
\end{equation}
and then the conclusion follows.
\end{proof}

\section{Main result}
\begin{lem}\label{mainlemma}
Let $(\XX,\dist,\mass)$ be an $\RCD(K,N)$ space of essential dimension $n$, let $F\in\BV_{loc}(\XX)^m$, let $\varphi\in \LIP(\RR^m)$  and let $v\in\TestV(\XX)$. Then for $|\DIFF F|$-a.e.\ $x\notin J_F$, $\varphi$ is differentiable at $F(x)$ in direction $(\nu_F \,\cdot \,v)(x)\in\RR^m$ and it holds

\begin{equation}\label{maineq}
	\nu_{\varphi\circ F} \,\cdot \,v|\DIFF(\varphi\circ F)|\mres (\XX\setminus J_F)=\nabla_{\nu_F\,\cdot\, v} \varphi(F) (\nu_F \,\cdot \, v){|{\DIFF F}|}\mres (\XX\setminus J_F).
\end{equation}

\end{lem}
\begin{proof}
By Proposition \ref{charts} and the fact that $|\DIFF(\varphi\circ F)|\le L|\DIFF F|$, where $L$ denotes the Lipschitz constant of $\varphi$, it is enough to prove the claim on $G$, where $G\subseteq\XX\setminus J_F$ is a bounded Borel set for which there exists a Borel map $\Psi:G\rightarrow B$, where $B$ is a Borel subset of $\RR^n$ and $\Psi$ has Borel inverse $\Psi^{-1}:B\rightarrow G$ and moreover $\Psi$ is the restriction to $G$ of some $\tilde{\Psi}\in\BV(\XX)^n$. Also, $J_{\tilde{\Psi}}\cap G=\emptyset$ and we can assume that $v$ has compact support.

We set then $F'\defeq(F,\tilde{\Psi})$ and $\varphi'\defeq\varphi\circ\pi$, where $\pi:\RR^m\times\RR^n\rightarrow\RR^m$ is the projection onto the first factor. In particular, $|\DIFF F|\le |\DIFF F'|$.  Notice that  $F'$ and $\varphi'$ still satisfy the assumptions of the lemma and that still $J_{F'}\cap G=\emptyset$. Notice also that $\varphi'\circ F'=\varphi\circ F$, that ($|\DIFF F|$-a.e.) $\varphi'$ is differentiable in direction $\nu_{F'} \,\cdot \,v$ if and only if $\varphi$ is differentiable in direction $\nu_F  \,\cdot \,v$ and finally that  $$(\nu_{F'} \,\cdot \,v)_i=(\nu_F \,\cdot \, v)_i\dv{|\DIFF F|}{|\DIFF F'|} \qquad|\DIFF F'|\text{-a.e.\ for $i=1,\dots,m$}$$
so that it remains to show \eqref{maineq} on $G$ with $\varphi'$ in place of $\varphi$ and $F'$ in place of $F$.

To simplify the notation, we return to the notation $F$ and $\varphi$, keeping in mind that $F$ is injective on $G$ and its inverse is Borel. As a notation, we set $$w\defeq (\nu_F \,\cdot \,v)\circ F^{-1}\qquad\text{and}\qquad\mu\defeq F_*(|\DIFF F|\mres G).$$

Assume for the moment also that $\varphi\in C^1(\RR^m)$. Then we know that  \eqref{maineq} holds with this choice of $\varphi$. We compute now 
\begin{equation}\notag
	F_*(\nabla\varphi(F)(\nu_F  \,\cdot \,v)|\DIFF F|\mres G)=\sum_{i=1}^m\partial_i\varphi F_*((\nu_F \,\cdot \,v)_i|\DIFF F|\mres G)=\nabla\varphi \,\cdot\,w  \mu.
\end{equation}
We check that the differentiation operator depending on $\varphi$ defined above is closable in the sense of item $iii)$ of Theorem \ref{AMthm}. We have to check that if $\{\varphi_k\}_k$ is a sequence as in item $iii)$ of Theorem \ref{AMthm} then there exists $\ell\in L^\infty(\mu)$ such that for every $h\in L^1(\mu)$,
\begin{align*}
	\int_{\RR^m}h \nabla\varphi_k\,\cdot\, w \dd{\mu}\rightarrow \int_{\RR^m} h \ell \dd{\mu}.
\end{align*}
Clearly, we can assume that $\varphi_k(0)=0$ for every $k$.
Equivalently, we have to prove that 
$$
\int_\XX h\circ F \nabla \varphi_k(F)(\nu_F \,\cdot \,v)\dd{|\DIFF F|\mres G}\rightarrow \int_\XX h\circ F\ell\circ F|\nu_F \,\cdot \,v|\dd{|\DIFF F|\mres G},
$$
where $h\circ F\in L^1(|\DIFF F|\mres G)$. As also $\varphi_k\in C^1(\RR^m)$, by \eqref{maineq} we have that 
$$ \int_\XX h\circ F\nabla\varphi_k(F)(\nu_F  \,\cdot \,v)\dd{|\DIFF F|\mres G}=\int_\XX h\circ F \nu_{\varphi_k\circ F} \,\cdot \,v\dd{|\DIFF (\varphi_k \circ F)|}\mres G,$$
which is well posed, since $|\DIFF(\varphi_k\circ F)|\le L |\DIFF F|$ for every $k$, where $L\in(0,\infty)$ denotes the Lipschitz constant of the functions in $\{\varphi_k\}_k$. Also, $|\DIFF(\varphi\circ F)|\le L |\DIFF F|$.
For every $\epsilon>0$, take $h_\epsilon\in \LIP_{bs}(\XX)$ such that $$\Vert h\circ F-h_\epsilon\Vert_{L^1(|\DIFF F|\mressmall A)}<\epsilon,$$
where we understand $h\circ F=0\ |\DIFF F|$-a.e.\ on $\XX\setminus G$ and $A$ is a neighbourhood of the support of $v$. 
By the Gauss--Green integration by parts formula, (with the usual interpretation of the integrals with $\dive(h_\epsilon v)\dd{\mass}$)
\begin{align*}
-\int_\XX& h_\epsilon\nu_{\varphi_k\circ F} \,\cdot \,v\dd{|\DIFF (\varphi_k \circ F)|}=\int_\XX \varphi_k\circ F\dive(h_\epsilon v)\dd{\mass}\\&\rightarrow \int_\XX \varphi\circ F\dive(h_\epsilon v)\dd{\mass}=-\int_\XX h_\epsilon\nu_{\varphi\circ F} \,\cdot \,v\dd{|\DIFF (\varphi \circ F)|}.
\end{align*}
Now,  we have that
$$ \abs{\int_\XX( h\circ F-h_\epsilon)\nu_{\varphi_k\circ F} \,\cdot \,v\dd{|\DIFF (\varphi_k \circ F)|}}\le L\Vert v\Vert_{L^\infty(T\XX)}\Vert h\circ F-h_\epsilon\Vert_{L^1(|\DIFF F|\mressmall A)}\le L\Vert v\Vert_{L^\infty(T\XX)}\epsilon, 
$$
and a similar estimate holds for $\varphi$ in place of $\varphi_k$. Then we see that 
\begin{align*}
	\int_\XX h\circ F \nabla \varphi_k(F)(\nu_F \,\cdot \,v)\dd{|\DIFF F|\mres G}\rightarrow &\int_\XX h\circ F\nu_{\varphi\circ F} \,\cdot \,v\dd{|\DIFF (\varphi \circ F)|}\mres G\\&=\int_{\RR^m}h (\nu_{\varphi\circ F} \,\cdot \,v)\circ F^{-1} \dv{|\DIFF (\varphi\circ F)|}{|\DIFF F|}\circ F^{-1}\dd{\mu}. 
\end{align*}
This provides the existence of the sought $\ell\in L^\infty(\mu)$.

Therefore we can apply Theorem \ref{AMthm}. It follows that if $\varphi$ is as in the statement, then $\varphi$ is differentiable in direction $w$ $\mu$-a.e. In other words, at $|\DIFF F|\mres G$-a.e.\ $x$, $\varphi$ is differentiable at $F(x)$ in direction $(\nu_F \,\cdot \, v)(x)$.

Take now $h\in L^1(|\DIFF F|\mres G)$. We approximate $\varphi$ with a sequence $\{\varphi_k\}_k$ as in Lemma \ref{ottopuntotre}. Using \eqref{maineq}, we see that for every $k$
$$
\int_\XX h\nu_{\varphi_k\circ F} \,\cdot \,v\dd{|\DIFF(\varphi_k\circ F)|\mres G}=\int_\XX h\nabla\varphi_k(F)(\nu_F \,\cdot \,v)\dd{|\DIFF F|\mres G}.
$$
Using dominated convergence to deal with the right hand side and by the very same computations as above to deal with the left hand side, we prove that \eqref{maineq} holds for $\varphi$.
\end{proof}

In view of the following result, recall the definition of $\essspan$, by Lemma \ref{essspan}.
\begin{thm}\label{gcthm}
Let $(\XX,\dist,\mass)$ be an $\RCD(K,N)$ space, let $F\in\BV_{loc}(\XX)^m$ and let $\varphi\in \LIP(\RR^m,\RR^l)$.
Then  
\begin{equation}\label{genjump}
		\nu_{\varphi\circ F}|\DIFF(\varphi\circ F)|\mres J_F=\frac{\varphi(F^r)-\varphi(F^l)}{|F^r-F^l|} \nu_F^J {|{\DIFF F}|}\mres  J_F.
\end{equation}
Set now $V\defeq |\DIFF F|-\essspan\nu_F$. Then for $|\DIFF F|$-a.e.\ $x\notin J_F$, $\varphi$ is differentiable at $F(x)$ with respect to $V$ and it holds
\begin{equation}\label{genjnump}
	\nu_{\varphi\circ F}|\DIFF(\varphi\circ F)|\mres (\XX\setminus J_F)=\nabla_V\varphi(F) \nu_F {|{\DIFF F}|}\mres (\XX\setminus J_F).
\end{equation}
\end{thm}
In the theorem above, by $\nabla_V\varphi(F) \nu_F $ we mean the unique, up to $|\DIFF F|$-a.e.\ equality, vector field in $L^0_\capa(T\XX)^l$ such that for every $v\in L^0_\capa(T\XX)$ it holds
$$
(\nabla_V\varphi(F) \nu_F ) \,\cdot \, v=\nabla_V\varphi(F) (\nu_F  \,\cdot \,v)\qquad  |\DIFF F|\text{-a.e.}
$$
\begin{proof}
	Denote by $n$ the essential dimension of $(\XX,\dist,\mass)$.
	It is easy to see that we can reduce ourselves to the case $l=1$.
	
We start from \eqref{genjump}. Notice first that $J_{\varphi\circ F}\subseteq J_F$. Take $\{\varphi_k\}_k\subseteq C^1(\RR^m)\cap\LIP(\RR^m)$ equi-Lipschitz, with $\varphi_k(0)=0$ for every $k$ and $\varphi_k\rightarrow\varphi$ locally uniformly.  Take any $v\in\TestV(\XX)$ with compact support.  

We know that 
$$
\int_\XX \nu_{\varphi_k\circ F} \,\cdot \,v\dd{|\DIFF(\varphi_k\circ F)|\mres J_F}=\int_\XX \frac{\varphi_k(F^r)-\varphi_k(F^l)}{\abs{F^r-F^l}}\nu_F^J \,\cdot \, v\dd{|\DIFF F|\mres J_F}.
$$
Therefore, as $k\rightarrow\infty$,
$$ \int_\XX \nu_{\varphi\circ F} \,\cdot \,v\dd{|\DIFF(\varphi\circ F)|\mres J_F}=\int_\XX \frac{\varphi(F^r)-\varphi(F^l)}{|{F^r-F^l}|}\nu_F^J \,\cdot \, v\dd{|\DIFF F|\mres J_F},$$
where convergence of the left hand side is proved as in the proof of Lemma \ref{mainlemma} whereas we use dominated convergence for the right hand side.
Being $v$ arbitrary, \eqref{genjump} follows.

We prove now the second part. We first show that for $|\DIFF F|$-a.e.\ $x\notin J_F$, $\varphi$ is differentiable at $F(x)$ with respect to $V$. Recalling the construction of $V$ in  Lemma \ref{essspan} (in particular, Theorem \ref{captangent}), it is enough to show this claim on a Borel subset $A$ on which we an orthonormal basis of $L^0_\capa(T\XX)$, say $\{v_1,\dots,v_k\}\subseteq\TestV(\XX)$: namely, we have to show differentiability at $F(x)$ with respect to $\langle\{\nu_F \,\cdot \, v_i\}_{i=1,\dots,k}\rangle$. 

By Lemma \ref{mainlemma}, if $v\in\TestV(\XX)$, for $|\DIFF F|$-a.e.\ $x\in A\setminus J_F$ it holds that $\varphi$ is differentiable at $F(x)$ in direction $\nu_F  \,\cdot \,v$. Therefore, for  $|\DIFF F|$-a.e.\ $x\in A\setminus J_F$
$\varphi$ is differentiable at $F(x)$ in every direction contained in $\langle\{\nu_F \,\cdot \, v_i\}_{i=1,\dots,k}\rangle_{\mathbb{Q}}$. Lemma \ref{mainlemma} again shows that the differential on $\langle\{\nu_F \,\cdot \, v_i\}_{i=1,\dots,k}\rangle_{\mathbb{Q}}$ is linear, up to discarding a set of null $|\DIFF F|\mres (A\setminus J_F)$ measure. 
It is then classical to infer from this the conclusion.
\end{proof}

\end{document}